\documentclass[10pt]{amsart}
\usepackage[leqno]{amsmath}
\usepackage{amssymb,latexsym,soul,cite,amsthm,color,enumitem,graphicx,mathtools,microtype,accents}
\usepackage[colorlinks=true,urlcolor=violet,citecolor=violet,linkcolor=violet,linktocpage,pdfpagelabels,bookmarksnumbered,bookmarksopen]{hyperref}
\definecolor{violet}{rgb}{0.56, 0.0, 1.0}
\usepackage[english]{babel}
\usepackage[left=2.5cm,right=2.5cm,top=2.5cm,bottom=2.5cm]{geometry}

\numberwithin{equation}{section}

\newtheorem{theorem}{Theorem}[section]
\theoremstyle{plain}

\theoremstyle{plain}

\theoremstyle{plain}
\newtheorem{corollary}[theorem]{Corollary}

\theoremstyle{definition}
\newtheorem{remark}[theorem]{Remark}

\newcommand{\N}{{\mathbb N}}

\newcommand{\R}{{\mathbb R}}
\newcommand{\eps}{\varepsilon}
\newcommand{\beq}{\begin{equation}}
\newcommand{\eeq}{\end{equation}}
\renewcommand{\le}{\leqslant}
\renewcommand{\ge}{\geqslant}

\makeatletter
\newcommand{\leqnomode}{\tagsleft@true}
\newcommand{\reqnomode}{\tagsleft@false}
\makeatother

\title[Clustering in fractional Sobolev spaces]{A clustering theorem in fractional Sobolev spaces}

\author[F.G.\ D\"uzg\"un, A.\ Iannizzotto, V.\ Vespri]{Fatma Gamze D\"uzg\"un, Antonio Iannizzotto, Vincenzo Vespri}

\address[F.G.\ D\"uzg\"un, A.\ Iannizzotto]
{Dipartimento di Matematica e Informatica
\newline\indent
Universit\`a degli Studi di Cagliari
\newline\indent
Via Ospedale 72, 09124 Cagliari, Italy}
\email{fatmagamze.duzgun@unica.it}
\email{antonio.iannizzotto@unica.it}

\address[V.\ Vespri]{Dipartimento di Matematica e Informatica ''U.\ Dini''
\newline\indent
Universit\`a degli Studi di Firenze
\newline\indent
Viale Morgagni 67/A, 50134 Firenze, Italy}
\email{vincenzo.vespri@unifi.it}


\begin{document}

\begin{abstract}
We prove a general clustering result for the fractional Sobolev space $W^{s,p}$: whenever the positivity set of a function $u$ in a square has measure bounded from below by a multiple of the cube's volume, and the $W^{s,p}$-seminorm of $u$ is bounded from above by a convenient power of the cube's side, then $u$ is positive in a universally reduced cube. Our result aims at applications in regularity theory for fractional elliptic and parabolic equations. Also, by means of suitable interpolation inequalities, we show that clustering results in $W^{1,p}$ and $BV$, respectively, can be deduced as special cases.
\end{abstract}

\maketitle

\begin{center}
Version of \today\
\end{center}

\section{Introduction and main result}\label{sec1}

\noindent
Clustering (or local clustering) is a general property shared by the weak solutions of several types of elliptic and parabolic equations, as well as functions in De Giorgi classes. It can basically described as follows. Let $u$ be a function defined in a cube $Q_r$ (with side $r>0$) and $c>0$ be a given level, satisfying the following conditions:
\begin{itemize}[leftmargin=1cm]
\item[$(a)$] the measure of the level set $Q_r\cap\{u>c\}$ is bounded from below by a multiple of the measure of the cube;
\item[$(b)$] the seminorm of $u$ in a certain function space with domain $Q_r$ is bounded from above by a multiple of $c$ times a convenient power of $r$. 
\end{itemize}
Then, for any $\lambda\in(0,1)$ the region $\{u>\lambda c\}$ 'clusters' at some point $x_1\in Q_r$, occupying a cube around $x_1$, whose side is proportional to $r$ by a constant independent of $u$.
\vskip2pt
\noindent
While condition $(a)$ and the conclusion are basically measure-theoretical properties, not affected by the regularity of the function $u$, condition $(b)$ may take different forms according to the space where we pick $u$ (which in turn depends on the differential or variational problem considered). Clustering results have been proved for $W^{1,p}(Q_r)$ ($p>1$) \cite{DV}, $W^{1,1}(Q_r)$ \cite{DGV}, and $BV(Q_r)$ \cite{TV}, each one with a different $(b)$-type condition, mainly by means of one-dimensional Poincar\'e inequalities. Some clustering theorems are stated on balls, rather than cubes, but essentially equivalent.
\vskip2pt
\noindent
In this note we consider the fractional Sobolev space $W^{s,p}(Q_r)$ with $s\in(0,1)$, $p\ge 1$, thus specializing the $(b)$-condition by means of the Gagliardo seminorm (see \cite{A,DPV,L} for an introduction to this class of function spaces). Precisely, for any open set $\Omega\subset\R^N$ ($N\ge 2$) we define for all measurable $u:\Omega\to\R$ the Gagliardo seminorm
\[[u]_{s,p,\Omega} = \Big[\iint_{\Omega\times\Omega} \frac{|u(x)-u(y)|^p}{|x-y|^{N+ps}}\,dx\,dy\Big]^\frac{1}{p}.\]
We say that $u\in W^{s,p}(\Omega)$, if $u\in L^p(\Omega)$ and $[u]_{s,p,\Omega}<\infty$. In order to state our result, for all $x\in\R^N$ and all $r>0$ let
\[Q_r(x) = \prod_{i=1}^N\Big(x^i-\frac{r}{2},\,x^i+\frac{r}{2}\Big)\]
be the $N$-dimensional open cube centered at $x$ with side $r$, so $|Q_r(x)|=r^N$ and ${\rm diag}(Q_r(x))=\sqrt{N}r$ (we denote $Q_r=Q_r(0)$, while $|A|$ always stands for the $N$-dimensional Lebesgue measure of any set $A\subset\R^N$). Our result is the following:

\begin{theorem}\label{main}
Let $s\in(0,1)$, $p\ge 1$, $x_0\in\R^N$, $r>0$ be given, and let $u\in W^{s,p}(Q_r(x_0))$, $c>0$ satisfy
\begin{itemize}[leftmargin=1cm]
\item[$(a)$] $\displaystyle\Big|\big\{x\in Q_r(x_0):\,u(x)>c\big\}\Big| > \alpha r^N$ $(\alpha\in(0,1))$;
\item[$(b)$] $[u]_{s,p,Q_r(x_0)} \le \gamma c r^\frac{N-ps}{p}$ $(\gamma>0)$.
\end{itemize}
Then, for all $\delta,\lambda\in(0,1)$ there exist $x_1\in Q_r(x_0)$, $\eta\in(0,1)$ (with $\eta$ independent of $u$) s.t.\
\[\Big|\big\{x\in Q_{\eta r}(x_1):\,u(x)>\lambda c\big\}\Big| > (1-\delta)(\eta r)^N.\]
\end{theorem}

\vskip2pt
\noindent
We display the proof in Section \ref{sec2}.
\vskip2pt
\noindent
Clustering results find applications in regularity theory for both elliptic and parabolic PDE's. To be more specific, in regularity theory one of the essential tools is the so-called {\em critical mass lemma} which states that, under suitable conditions, there exists an absolute constant $\nu\in(0,1)$ s.t.\ whenever $u$ is a weak solution of an elliptic or parabolic equation, $\mu$ is the infimum of $u$ in a ball $B_r\subset\R^N$ with radius $r>0$, and the measure of the set $\{u>\mu+\eps\}$ is larger than $\nu\omega_N r^N$ (where $\eps>0$ and $\omega_N>0$ denotes the volume of the $N$-dimensional unit ball), then $u\ge\mu+\eps/2$ a.e.\ in $B_{r/2}$. Such information is extremely important, as it marks out a region of positivity, which in turn plays a fundamental role in the proofs of H\"older continuity and Harnack estimates for $u$. Now the clustering theorem ensures that, provided $u$ has enough regularity, there is some part of the domain where the hypotheses of the critical mass lemma are satisfied. This method was introduced in \cite{DV} and then applied to nonlinear PDE's, both in the elliptic \cite{DMV} and the parabolic case \cite{DMV1}. Also, clustering is an essential tool for proving Harnack's inequality without continuity in anisotropic problems, see \cite[p.\ 370]{D}.
\vskip2pt
\noindent
The main motivation for our result is related to a reinterpretation of regularity theory for nonlinear, nonlocal equations driven by the $s$-fractional $p$-Laplacian, properly defined as the gradient of the functional
\[W^{s,p}(\R^N) \ni u\mapsto \frac{[u]_{s,p,\R^N}^p}{p}.\]
Such operator was introduced in \cite{IN} as an approximation of the $p$-Laplacian (for $s\to 1$) and in \cite{LL} as an approximation of the fractional infinity Laplacian (for $p\to\infty$), see also \cite{BCF} for nonlinear fractional operators arising from game theory. When $p\ge 2$ and $u$ is very smooth, the fractional $p$-Laplacian admits an alternative, pointwise representation as
\[(-\Delta)_p^s u(x) = 2\lim_{\eps\to 0^+}\int_{B^c_\eps(x)}\frac{|u(x)-u(y)|^{p-2}(u(x)-u(y))}{|x-y|^{N+ps}}\,dy.\]
H\"older continuity and Harnack's inequality have been established for the fractional $p$-Laplacian in the elliptic case in \cite{DKP,DKP1} through an adapted De Giorgi-Moser-Nash approach (see also \cite{C} for an alternative approach based on fractional De Giorgi classes), and higher H\"older regularity is proved, for instance, in the recent paper \cite{BDLMBS}. On the parabolic side, regularity results for the evolutive fractional $p$-Laplace equation have been established in \cite{L1,DZZ,V,V1}.
\vskip2pt
\noindent
An application of Theorem \ref{main} has already appeared in \cite{CDI}, where local clustering is employed, along with positivity expansion and a fractional critical mass lemma, to give an alternative proof of H\"older regularity for the solutions of the fractional $p$-Laplace equation. Also, we aim at proving classical Harnack inequalities for the corresponding evolutive equation in the singular case ($1<p<2$). A secondary, but hopefully interesting, motivation of the present note is that, in fact, previous clustering theorems in classical Sobolev spaces $W^{1,p}$ (see \cite{DV} for $p>1$, \cite{DGV} for $p=1$) and the bounded variation space $BV$ (see \cite{TV}) can be deduced from Theorem \ref{main}, through convenient interpolation and scaling inequalities. So, we design here a unified approach to the problem of clustering. In this connection, we present a very simple proof of known interpolation inequalities between $W^{1,p}$ (resp., $BV$) and $W^{s,p}$ with a pure seminorm control (see Remark \ref{emb}). Section \ref{sec3} is devoted to this subject. We mention that a similar clustering result for the linear fractional Laplacian, with a more general kernel, has been proved independently from us in \cite{C1}.

\section{Proof of Theorem \ref{main}}\label{sec2}

\noindent
First we assume that $u\in C(\overline{Q_r(x_0)})$. For any $k\in\N$, we partition $Q_r(x_0)$ into a family $\mathcal{F}_k$ of cubes with side $r/k$, so that for all $Q\in\mathcal{F}_k$ we have
\[|Q| = \Big(\frac{r}{k}\Big)^N, \ {\rm diag}(Q) = \frac{\sqrt{N}r}{k} = \sqrt{N}|Q|^\frac{1}{N}.\]
Clearly $\#\mathcal{F}_k=k^N$. Let $\alpha\in(0,1)$ be as in $(a)$ and set for all $Q\in\mathcal{F}_k$
\[\begin{cases}
Q\in\mathcal{F}_k^+ & \text{if $\displaystyle\Big|Q\cap\{u>c\}\Big|\ge\frac{\alpha}{2}|Q|$} \\
Q\in\mathcal{F}_k^- & \text{otherwise.}
\end{cases}\]
By hypothesis $(a)$, for all $Q\in\mathcal{F}_k^+$ we have
\begin{align*}
\alpha k^N|Q| &= \alpha|Q_r(x_0)| \\
&< \Big|Q_r(x_0)\cap\{u>c\}\Big| \\
&= \sum_{Q\in\mathcal{F}_k^+}\Big|Q\cap\{u>c\}\Big|+\sum_{Q\in\mathcal{F}_k^-}\Big|Q\cap\{u>c\}\Big|.
\end{align*}
By definition of $\mathcal{F}_k^\pm$, then,
\begin{align*}
\alpha k^N &< \sum_{Q\in\mathcal{F}_k^+}\frac{\big|Q\cap\{u>c\}\big|}{|Q|}+\sum_{Q\in\mathcal{F}_k^-}\frac{\big|Q\cap\{u>c\}\big|}{|Q|} \\
&\le \#\mathcal{F}_k^++\frac{\alpha}{2}\#\mathcal{F}_k^- \\
&= \frac{\alpha}{2}k^N+\Big(1-\frac{\alpha}{2}\Big)\#\mathcal{F}_k^+,
\end{align*}
which rephrases as
\beq\label{clu1}
\#\mathcal{F}_k^+ > \frac{\alpha}{2-\alpha}k^N.
\eeq
Now fix $\delta,\lambda\in(0,1)$. Two cases may occur:
\begin{itemize}[leftmargin=1cm]
\item[$(i)$] There exist $k\ge 2$, $Q\in\mathcal{F}_k^+$ s.t.\
\[\Big|Q\cap\{u>\lambda c\}\Big| > (1-\delta)|Q|.\]
Then, let $x_1\in Q_r(x_0)$ be the center of $Q$, $\eta=1/k\in(0,1)$, and the conclusion follows.
\item[$(ii)$] For all $k\ge 2$ and all $Q\in\mathcal{F}_k^+$ we have
\[\Big|Q\cap\{u>\lambda c\}\Big| \le (1-\delta)|Q|,\]
which is equivalent to
\beq\label{clu2}
\Big|Q\cap\{u\le\lambda c\}\Big| \ge \delta|Q|.
\eeq
Also, since $\lambda\in(0,1)$ and $Q\in\mathcal{F}_k^+$, we have
\beq\label{clu3}
\Big|Q\cap\Big\{u>\frac{\lambda+1}{2}\,c\Big\}\Big| \ge \frac{\alpha}{2}\,|Q|.
\eeq
Fix now $x,y\in Q$ s.t.\
\[u(x) \le \lambda c, \ u(y) > \frac{\lambda+1}{2}\,c,\]
hence
\[u(y)-u(x) > \frac{1-\lambda}{2}\,c.\]
By continuity of $u$, we have $|x-y|\ge\mu$ for some $\mu>0$ independent of $x$, $y$ (which makes all of the following integrals nonsingular). Now we start from \eqref{clu3} and integrate with respect to $y$, then we use H\"older's inequality:
\begin{align*}
\frac{\alpha(1-\lambda)}{4}\,c|Q| &\le \frac{1-\lambda}{2}\,c\Big|Q\cap\Big\{u>\frac{\lambda+1}{2}\,c\Big\}\Big| \\
&= \int_{Q\cap\{u>\frac{\lambda+1}{2}c\}}\frac{1-\lambda}{2}\,c\,dy \\
&\le \int_{Q\cap\{u>\frac{\lambda+1}{2}c\}}|u(x)-u(y)|\,dy \\
&\le {\rm diag}(Q)^\frac{N+ps}{p} \int_{Q\cap\{u>\frac{\lambda+1}{2}c\}}\frac{|u(x)-u(y)|}{|x-y|^\frac{N+ps}{p}}\,dy \\
&\le N^\frac{N+ps}{2p}|Q|^\frac{N+ps}{Np} \Big[\int_{Q\cap\{u>\frac{\lambda+1}{2}c\}}\frac{|u(x)-u(y)|^p}{|x-y|^{N+ps}}\,dy\Big]^\frac{1}{p}\Big|Q\cap\Big\{u>\frac{\lambda+1}{2}\,c\Big\}\Big|^\frac{p-1}{p} \\
&\le N^\frac{N+ps}{2p}|Q|^\frac{N+s}{N} \Big[\int_{Q\cap\{u>\frac{\lambda+1}{2}c\}}\frac{|u(x)-u(y)|^p}{|x-y|^{N+ps}}\,dy\Big]^\frac{1}{p},
\end{align*}
hence for all $x\in Q\cap\{u\le\lambda c\}$ we have
\beq\label{clu4}
\Big[\frac{\alpha(1-\lambda)c}{4N^\frac{N+ps}{2p}}\Big]^p |Q|^{-\frac{ps}{N}} \le \int_{Q\cap\{u>\frac{\lambda+1}{2}c\}}\frac{|u(x)-u(y)|^p}{|x-y|^{N+ps}}\,dy.
\eeq
Next we begin with \eqref{clu2}, use \eqref{clu4} and integrate with respect to $x$:
\begin{align*}
\Big[\frac{\alpha(1-\lambda)c}{4N^\frac{N+ps}{2p}}\Big]^p \delta|Q|^\frac{N-ps}{N} &\le \Big|Q\cap\{u\le\lambda c\}\Big| \int_{Q\cap\{u>\frac{\lambda+1}{2}c\}}\frac{|u(x)-u(y)|^p}{|x-y|^{N+ps}}\,dy \\
&\le \int_{Q\cap\{u\le\lambda c\}}\int_{Q\cap\{u>\frac{\lambda+1}{2}c\}}\frac{|u(x)-u(y)|^p}{|x-y|^{N+ps}}\,dy\,dx \\
&\le \iint_{Q\times Q}\frac{|u(x)-u(y)|^p}{|x-y|^{N+ps}}\,dx\,dy.
\end{align*}
Under our current assumptions, the inequality above holds for all $Q\in\mathcal{F}_k^+$. Further, we sum over $Q\in\mathcal{F}_k^+$ and use \eqref{clu1} along with hypothesis $(b)$:
\begin{align*}
\frac{\alpha}{2-\alpha}\,k^N \Big[\frac{\alpha(1-\lambda)c}{4N^\frac{N+ps}{2p}}\Big]^p \delta|Q|^\frac{N-ps}{N} &\le \sum_{Q\in\mathcal{F}_k^+}\iint_{Q\times Q}\frac{|u(x)-u(y)|^p}{|x-y|^{N+ps}}\,dx\,dy \\
&\le \iint_{Q_r(x_0)\times Q_r(x_0)}\frac{|u(x)-u(y)|^p}{|x-y|^{N+ps}}\,dx\,dy \\
&\le \gamma^p c^p r^{N-ps}.
\end{align*}
Recalling that $|Q|=(r/k)^N$, we get for all $k\ge 2$
\[\frac{\alpha^{p+1}(1-\lambda)^p\delta}{4^p(2-\alpha)N^\frac{N+ps}{2}}\,r^{N-ps}\,k^{ps} \le \gamma^p r^{N-ps},\]
hence
\[k^{ps} \le \frac{4^p(2-\alpha)N^\frac{N+ps}{2}\gamma^p}{\alpha^{p+1}(1-\lambda)^p\delta}.\]
Letting $k\to\infty$ we find a contradiction.
\end{itemize}
Thus, case $(i)$ must occur for some $k\ge 2$, which proves the assertion for continuous functions.
\vskip2pt
\noindent
Now let $u\in W^{s,p}(Q_r(x_0))$ be arbitrary. By \cite[Theorems 2.4, 5.4]{DPV} we can find a sequence $(u_n)$ in $C^\infty(\overline{Q_r(x_0)})$ s.t.\ $u_n\to u$ in $W^{s,p}(Q_r(x_0))$. In particular we have $u_n(x)\to u(x)$ for a.e.\ $x\in Q_r(x_0)$ and
\[\lim_n [u_n]_{s,p,Q_r(x_0)} = [u]_{s,p,Q_r(x_0)}.\]
Fix an arbitrary $\tilde\gamma>\gamma$, then by $(b)$ and the previous convergence we have for all $n\in\N$ big enough
\beq\label{clu5}
[u_n]_{s,p,Q_r(x_0)} < \tilde\gamma cr^\frac{N-ps}{p}.
\eeq
For a.e.\ $x\in Q_r(x_0)\cap\{u>c\}$ we have
\[\chi_{Q_r(x_0)\cap\{u_n>c\}}(x) \to 1\]
(henceforth, $\chi_A$ denotes the characteristic function of any set $A\subset\R^N$). By Fatou's lemma, then,
\begin{align*}
\Big|Q_r(x_0)\cap\{u>c\}\Big| &= \int_{Q_r(x_0)\cap\{u>c\}}1\,dx \\
&\le \liminf_n\int_{Q_r(x_0)\cap\{u>c\}}\chi_{Q_r(x_0)\cap\{u_n>c\}}(x)\,dx \\
&\le \liminf_n\Big|Q_r(x_0)\cap\{u_n>c\}\Big|.
\end{align*}
By $(a)$, for all $n\in\N$ big enough we have
\beq\label{clu6}
\Big|Q_r(x_0)\cap\{u_n>c\}\Big| > \alpha r^N.
\eeq
Fix now $\delta,\lambda\in (0,1)$, and pick any $\tilde\delta\in(0,\delta)$, $\tilde\lambda\in(\lambda,1)$. As in the previous case, for all $n\in\N$ big enough by \eqref{clu5} and \eqref{clu6} there exist $x_n\in Q_r(x_0)$ and a number $\eta\in(0,1)$ (independent of $n$) s.t.\
\[\Big|Q_{\eta r}(x_n)\cap\{u_n>\tilde\lambda c\}\Big| > (1-\tilde\delta)(\eta r)^N.\]
In fact, as seen before $\eta=1/k$, and $Q_{\eta r}(x_n)$ is one of the fixed $k^N$ cubes of the family $\mathcal{F}_k$. Passing if necessary to a subsequence, we may assume that $x_n$ is the same for all $n\in\N$. Let us denote it $x_1$, and set $Q=Q_{\eta r}(x_1)$, so for all $n\in\N$
\beq\label{clu7}
\Big|Q\cap\{u_n<\tilde\lambda c\}\Big| < \tilde\delta(\eta r)^N.
\eeq
As above we have $\chi_{Q\cap\{u_n<\tilde\lambda c\}}\to 1$ a.e.\ in $Q\cap\{u<\tilde\lambda c\}$, hence by Fatou's lemma and \eqref{clu7} we have
\begin{align*}
\Big|Q\cap\{u<\tilde\lambda c\}\Big| &\le \liminf_n\Big|Q\cap\{u_n<\tilde\lambda c\}\Big| \\
&\le \tilde\delta(\eta r)^N.
\end{align*}
Reversing the inequality we get
\[\Big|Q\cap\{u\ge\tilde\lambda c\}\Big| \ge (1-\tilde\delta)(\eta r)^N,\]
so recalling that $\tilde\lambda>\lambda$ and $\tilde\delta<\delta$, we have
\[\Big|Q\cap\{u>\lambda c\}\Big| > (1-\delta)(\eta r)^N,\]
which concludes the proof. \qed

\section{Special cases}\label{sec3}

\noindent
In this section we consider two special cases, respectively, $u\in W^{1,p}(Q_r(x_0))$ ($p\ge 1$) and $u\in BV(Q_r(x_0))$, previously studied in the literature. Such cases can be reduced to our framework by means of convenient interpolation inequalities (see \cite{BBM,P}), for which we present direct proofs.
\vskip2pt
\noindent
We begin with the case $W^{1,p}$, setting for all open $\Omega$ and all $u\in W^{1,p}(\Omega)$
\[\|\nabla u\|_{p,\Omega} = \Big[\int_\Omega|\nabla u(x)|^p\,dx\Big]^\frac{1}{p}.\]
The following result is equivalent to \cite[Proposition A.1]{DV} ($p>1$) and \cite{DGV} ($p=1$):

\begin{corollary}\label{sob}
Let $p\ge 1$, $x_0\in\R^N$, $r>0$ be given, and let $u\in W^{1,p}(Q_r(x_0))$, $c>0$ satisfy $(a)$ and
\begin{itemize}[leftmargin=1cm]
\item[$(b')$] $\|\nabla u\|_{p,Q_r(x_0)} \le \gamma' c r^\frac{N-p}{p}$ ($\gamma'>0$).
\end{itemize}
Then, the conclusion of Theorem \ref{main} holds.
\end{corollary}
\begin{proof}
In view of Theorem \ref{main}, we only need to show that $u\in W^{s,p}(Q_r(x_0))$ satisfies $(b)$, with a possibly different $\gamma>0$ independent of $u$. Without loss of generality we may assume $x_0=0$.
\vskip2pt
\noindent
First we prove that for all $s\in(0,1)$ there exists $C=C(N,s,p)>0$ s.t.\ for all $v\in W^{1,p}(Q_1)$
\beq\label{sob1}
[v]_{s,p,Q_1} \le C\|\nabla v\|_{p,Q_1}.
\eeq
First assume $v\in C^1(\overline{Q_1})$ and set $\sigma=N+ps-p<N$. Integrating on segments we have
\begin{align*}
\iint_{Q_1\times Q_1}\frac{|v(x)-v(y)|^p}{|x-y|^{N+ps}}\,dx\,dy &= \iint_{Q_1\times Q_1}\Big|\int_0^1 \nabla v(x+t(y-x))\cdot(y-x)\,dt\Big|^p\,\frac{dx\,dy}{|x-y|^{N+ps}} \\
&\le \iint_{Q_1\times Q_1}\int_0^1 \big|\nabla v(x+t(y-x))\big|^p\,dt\Big[\int_0^1 |y-x|^{p'}\,dt\Big]^{p-1}\,\frac{dx\,dy}{|x-y|^{N+ps}} \\
&= \iint_{Q_1\times Q_1}\int_0^1 \big|\nabla v(x+t(y-x))\big|^p\,dt\,\frac{dx\,dy}{|x-y|^\sigma} \\
&= \int_0^1\,\iint_{\R^N\times\R^N}\frac{|\nabla v(x+t(y-x))|^p}{|x-y|^\sigma}\chi_{Q_1\times Q_1}(x,y)\,dx\,dy\,dt.
\end{align*}
Fix $t\in[0,1]$ and set $z=y-x$, $w=x+t(y-x)$. By convexity of $Q_1$ we have
\[\begin{cases}
x = w-tz \in Q_1 \\
y = w+z-tz \in Q_1
\end{cases} \ \Longrightarrow \ 
\begin{cases}
w = (1-t)(w-tz)+t(w+z-tz) \in Q_1 \\
z = (w+z-tz)-(w-tz) \in Q_2.
\end{cases}\]
So for all $t\in[0,1]$ we have
\begin{align*}
\iint_{\R^N\times\R^N}\frac{|\nabla v(x+t(y-x))|^p}{|x-y|^\sigma}\chi_{Q_1\times Q_1}(x,y)\,dx\,dy &= \iint_{\R^N\times\R^N}\frac{|\nabla v(w)|^p}{|z|^\sigma}\chi_{Q_1\times Q_1}(w-tz,\,w+z-tz)\,dw\,dz \\
&\le \int_{Q_2}\frac{dz}{|z|^\sigma}\,\int_{Q_1}|\nabla v(w)|^p\,dw = C\|\nabla v\|_{p,Q_1}^p,
\end{align*}
with $C>0$ only depending on $N,s,p$. We next integrate with respect to $t$ and find
\[\iint_{Q_1\times Q_1}\frac{|v(x)-v(y)|^p}{|x-y|^{N+ps}}\,dx\,dy \le C\|\nabla v\|_{p,Q_1}^p,\]
which proves \eqref{sob1}. If $v\in W^{1,p}(Q_1)$ is arbitrary, then we can find a sequence $(v_n)$ in $C^\infty(\overline{Q_1})$ s.t.\ $v_n\to v$ in both $W^{1,p}(Q_1)$ and $W^{s,p}(Q_1)$. For all $n\in\N$ we have
\[[v_n]_{s,p,Q_1} \le C\|\nabla v_n\|_{p,Q_1},\]
so passing to the limit we get \eqref{sob1}.
\vskip2pt
\noindent
Now we recall some useful scaling formulas. Let $u\in W^{1,p}(Q_r)$ be as in the assumption. Setting $v(x)=u(rx)$ for all $x\in Q_1$, we have $v\in W^{1,p}(Q_1)$ and
\beq\label{sob2}
\|\nabla u\|_{p,Q_r} = r^\frac{N-p}{p}\|\nabla v\|_{p,Q_1},
\eeq
\beq\label{sob3}
[u]_{s,p,Q_r} = r^\frac{N-ps}{p}[v]_{s,p,Q_1}.
\eeq
Concatenating \eqref{sob3}, \eqref{sob1}, \eqref{sob2}, and hypothesis $(b')$ we get
\begin{align*}
[u]_{s,p,Q_r} &= r^\frac{N-ps}{p}[v]_{s,p,Q_1} \\
&\le Cr^\frac{N-ps}{p}\|\nabla v\|_{p,Q_1} \\
&= Cr^{1-s}\|\nabla u\|_{p,Q_r} \le (C\gamma')cr^\frac{N-ps}{p}.
\end{align*}
So $u$ satisfies $(b)$ with $\gamma=C\gamma'>0$ (independent of $u$). The conclusion now follows from Theorem \ref{main}.
\end{proof}

\noindent
Finally we consider the case $BV$. We recall that $u\in BV(\Omega)$ if $u\in L^1(\Omega)$ and the quantity
\[[u]_{BV(\Omega)} = \sup\Big\{\int_\Omega u\,{\rm div}\,\phi\,dx:\,\phi\in C^\infty_c(\Omega,\R^N),\,|\phi(x)|\le 1 \ \text{for all $x\in\Omega$}\Big\}\]
is finite. The following result is equivalent to \cite[Lemma 1.1]{TV}:

\begin{corollary}\label{bvo}
Let $x_0\in\R^N$, $r>0$ be given, and let $u\in BV(Q_r(x_0))$, $c>0$ satisfy $(a)$ and
\begin{itemize}[leftmargin=1cm]
\item[$(b'')$] $[u]_{BV(Q_r(x_0))} \le \gamma'' cr^{N-1}$ ($\gamma''>0$).
\end{itemize}
Then, the conclusion of Theorem \ref{main} holds.
\end{corollary}
\begin{proof}
As in the previous case, we assume $x_0=0$ and fix $s\in(0,1)$. First we see that there exists $C=C(N,s)>0$ s.t.\ for all $v\in BV(Q_1)$
\beq\label{bvo1}
[v]_{s,1,Q_1} \le C[v]_{BV(Q_1)}.
\eeq
Indeed, by classical density results (see for instance \cite[Theorem 1.17]{G}), there exists a sequence $(v_n)$ in $C^\infty(\overline{Q_1})$ s.t.\ $v_n\to v$ in $L^1(Q_1)$ and
\[\lim_n\|\nabla v_n\|_{1,Q_1} = [v]_{BV(Q_1)}.\]
By \eqref{sob1} we have for all $n\in\N$ and some $C>0$ independent of $n$
\[[v_n]_{s,1,Q_1} \le C\|\nabla v_n\|_{1,Q_1}.\]
Also, up to a subsequence $v_n\to v$ in $W^{s,1}(Q_1)$. So we can pass to the limit as $n\to\infty$ and find \eqref{bvo1}. In particular, then, for all $s\in(0,1)$ we see that $BV(Q_1)\subseteq W^{s,1}(Q_1)$ with continuous embedding.
\vskip2pt
\noindent
Let $u\in BV(Q_r)$ be as in the assumption. Setting $v(x)=u(rx)$ for all $x\in Q_1$, we have $v\in L^1(Q_1)$ and the scaling formula
\beq\label{bvo2}
[u]_{BV(Q_r)} = r^{N-1}[v]_{BV(Q_1)}.
\eeq
Indeed, fix $\phi\in C^\infty_c(Q_r,\R^N)$ s.t.\ $|\phi|\le 1$ in $\in Q_r$ and set $\psi(y)=\phi(ry)$ for all $y\in Q_1$, then $\psi\in C^\infty_c(Q_1,\R^N)$ and $|\psi|\le 1$ in $Q_1$. Moreover,
\begin{align*}
\int_{Q_r} u(x)\,{\rm div}\,\phi(x)\,dx &= r^N\int_{Q_1}u(ry)\,{\rm div}\,\phi(ry)\,dy \\
&= r^{N-1}\int_{Q_1}v(y)\,{\rm div}\,\psi(y)\,dy.
\end{align*}
Taking the suprema over $\phi$, $\psi$, respectively, we have \eqref{bvo2}. Now, using \eqref{sob3} (with $p=1$), \eqref{bvo1}, \eqref{bvo2}, and hypothesis $(b'')$ we have
\begin{align*}
[u]_{s,1,Q_r} &= r^{N-s}[v]_{s,1,Q_1} \\
&\le Cr^{N-s}[v]_{BV(Q_1)} \\
&= Cr^{1-s}[u]_{BV(Q_r)} \le (C\gamma'')cr^{N-s}.
\end{align*}
Therefore, $u$ satisfies $(b)$ with $p=1$ and $\gamma=C\gamma''>0$ (independent of $u$). The conclusion now follows from Theorem \ref{main}.
\end{proof}

\begin{remark}\label{emb}
A brief discussion about inequalities \eqref{sob1}, \eqref{bvo1} is in order. Note that in both inequalities we control a seminorm by means of another seminorm, which is a sharper result than usual embedding theorems involving full norms (which incorporate a $L^p$-norm as well). Inequality \eqref{sob1} is essentially contained in the proof of \cite[Theorem 1]{BBM}, where it is obtained, for a 'smooth' domain, via a seminorm-preserving extension operator $W^{1,p}(\Omega)\to W^{1,p}(\R^N)$. Extension operators of this type have been detected for connected, Lipschitz domains in \cite{B,B1} (see also \cite{J} for a more general class of Lipschitz domains). Besides, an extension operator for our cubic domain $Q_1$ can also be obtained by reflection. Nevertheless, we included our proof of \eqref{sob1} because it is very simple and {\em does not} involve any extension procedure, in addition it is easily adapted to any bounded, convex domain. A similar discussion applies to \eqref{bvo1} (see also \cite{BNLT} for the relation between bounded variation and fractional Sobolev functions in dimension one).
\end{remark}

\vskip4pt
\noindent
{\bf Acknowledgement.} The authors are members of GNAMPA (Gruppo Nazionale per l'Analisi Matematica, la Probabilit\`a e le loro Applicazioni) of INdAM (Istituto Nazionale di Alta Matematica 'Francesco Severi'). A.\ Iannizzotto is partially supported by the research project {\em Problemi non locali di tipo stazionario ed evolutivo} (GNAMPA, CUP E53C23001670001). The authors are grateful to S.\ Mosconi for stimulating discussions and to P.D.\ Lamberti for useful references on the extension problem. Finally, the authors are grateful to the anonymous Referee for her/his careful reading of this paper and kind words of appreciation.


\begin{thebibliography}{99}

\bibitem{A}
{\sc R.A.\ Adams,}
Sobolev spaces,
Pure and Applied Mathematics, vol.\ 65, Academic Press, New York (1975).

\bibitem{BNLT}
{\sc M.\ Bergounioux, A.\ Leaci, G.\ Nardi, F.\ Tomarelli,}
Fractional Sobolev spaces and functions of bounded variation of one variable,
{\em Fract. Calc. Appl. Anal.} {\bf 20} (2017) 936--962.

\bibitem{BCF}
{\sc C.\ Bjorland, L.\ Caffarelli, A.\ Figalli,}
Non-local gradient dependent operators,
{\em Adv. Math.} {\bf 230} (2012) 1859--1894.

\bibitem{BDLMBS}
{\sc V.\ B\"ogelein, F.\ Duzaar, N.\ Liao, G.\ Molica Bisci, R.\ Servadei,}
Regularity for the fractional $p$-Laplace equation,
preprint (arXiv:2406.01568).

\bibitem{BBM}
{\sc J.\ Bourgain, H.\ Brezis, P.\ Mironescu,}
Another look at Sobolev spaces,
Optimal control and partial differential equations, IOS, Amsterdam (2001).

\bibitem{B}
{\sc V.\ Burenkov,}
The extension of functions with preservation of the seminorm,
{\em Dokl. Akad. Nauk SSSR} {\bf 228} (1976) 779--782.

\bibitem{B1}
{\sc V.\ Burenkov,}
Extension of functions with preservation of the Sobolev seminorm,
{\em Trudy Mat. Inst. Steklov} {\bf 172} (1985) 71--85.

\bibitem{CDI}
{\sc F.M.\ Cassanello, F.G.\ D\"uzg\"un, A.\ Iannizzotto,}
H\"older regularity for the fractional $p$-Laplacian, revisited,
{\em Adv. Calc. Var.} (2025) (https://doi.org/10.1515/acv-2024-0103).

\bibitem{C1}
{\sc J.\ Chen,}
H\"older and Harnack estimates for integro-differential operators with kernels of measure,
{\em Ann. Mat. Pura Appl.} (to appear).

\bibitem{C}
{\sc M.\ Cozzi,}
Regularity results and Harnack inequalities for minimizers and solutions of nonlocal problems: a unified approach via fractional De Giorgi classes,
{\em J. Funct. Anal.} {\bf 272} (2017) 4762--4837.

\bibitem{D}
{\sc E.\ DiBenedetto,}
Partial differential equations,
{\em Cornerstones}, Birkh\"auser, Boston (2010), p.\ 389.

\bibitem{DGV}
{\sc E.\ DiBenedetto, U.\ Gianazza, V.\ Vespri,}
Local clustering of the non-zero set of functions in $W^{1,1}(E)$,
{\em Atti Accad. Naz. Lincei Rend. Lincei Mat. Appl.} {\bf 17} (2006) 223--225.

\bibitem{DV}
{\sc E.\ DiBenedetto, V.\ Vespri,}
On the singular equation $\beta(u)_t=\Delta u$,
{\em Arch. Rational Mech. Anal.} {\bf 132} (1995) 247--309.

\bibitem{DKP}
{\sc A.\ Di Castro, T.\ Kuusi, G.\ Palatucci,}
Local behavior of fractional $p$-minimizers,
{\em Ann. Inst. H. Poincar\'e C Anal. Non Lin\'eaire} {\bf 33} (2016) 1279--1299.

\bibitem{DKP1}
{\sc A.\ Di Castro, T.\ Kuusi, G.\ Palatucci,}
Nonlocal Harnack inequalities,
{\em J. Funct. Anal.} {\bf 267} (2014) 1807--1836.

\bibitem{DPV}
{\sc E.\ Di Nezza, G.\ Palatucci, E.\ Valdinoci},
Hitchhiker's guide to the fractional Sobolev spaces,
{\em Bull. Sci. Math.} {\bf 136} (2012) 521--573.

\bibitem{DZZ}
{\sc M.\ Ding, C.\ Zhang, S.\ Zhou,}
Local boundedness and H\"older continuity for the parabolic fractional $p$-Laplace equations,
{\em Calc. Var. Partial Differential Equations} {\bf 60} (2021) art.\ 38. 

\bibitem{DMV}
{\sc F.G.\ D\"uzg\"un, P.\ Marcellini, V.\ Vespri,}
An alternative approach to the H\"older continuity of solutions to some elliptic equations,
{\em Nonlinear Anal.} {\bf 94} (2014) 133--141.

\bibitem{DMV1}
{\sc F.G.\ D\"uzg\"un, S.\ Mosconi, V.\ Vespri,} Harnack and pointwise estimates for degenerate or singular parabolic equations, in {\em Contemporary research in elliptic PDEs and related topics}, Springer-INdAM Series Vol.\ 33, Springer, Cham (2019) 301--368.

\bibitem{G}
{\sc E.\ Giusti,}
Minimal surfaces and functions of bounded variation,
{\em Monographs in Math.}, Birkh\"auser, Basel (1984), p.\ 240.

\bibitem{IN}
{\sc H.\ Ishii, G.\ Nakamura,}
A class of integral equations and approximation of $p$-Laplace equations,
{\em Calc. Var. Partial Differential Equations} {\bf 37} (2010) 485--522.

\bibitem{J}
{\sc P.W.\ Jones,}
Quasiconformal mappings and extendability of functions in Sobolev spaces,
{\em Acta Math.} {\bf 147} (1981) 71--88.

\bibitem{L}
{\sc G.\ Leoni,}
A first course in fractional Sobolev spaces,
American Mathematical Society, Providence (2023).

\bibitem{L1}
{\sc N.\ Liao,}
H\"older regularity for parabolic fractional $p$-Laplacian,
{\em Calc. Var. Partial Differential Equations} {\bf 63} (2024) art.\ 22.

\bibitem{LL}
{\sc E.\ Lindgren, P.\ Lindqvist,}
Fractional eigenvalues,
{\em Calc. Var. Partial Differential Equations} {\bf 49} (2014) 795--826.

\bibitem{P}
{\sc A.C.\ Ponce,}
A new approach to Sobolev spaces and connections to $\Gamma$-convergence,
{\em Calc. Var.} {\bf 19} (2004) 229--255.

\bibitem{TV}
{\sc A.\ Telcs, V.\ Vespri,}
A quantitative Lusin theorem for functions in $BV$, in {\em Geometric methods in PDE's}, Springer-INdAM Series Vol.\ 13, Springer, Cham (2015) 81--87.

\bibitem{V}
{\sc J.L.\ V\'azquez,}
The evolution fractional $p$-Laplacian equation in $\R^N$. Fundamental solution and asymptotic behaviour,
{\em Nonlinear Anal.} {\bf 199} (2020) art.\ 112034.

\bibitem{V1}
{\sc J.L.\ V\'azquez,}
The fractional $p$-Laplacian evolution equation in $\R^N$ in the sublinear case,
{\em Calc. Var.} {\bf 60} (2021) art.\ 140.

\end{thebibliography}
\end{document}